\documentclass[12pt]{amsart}
\usepackage{amssymb,latexsym,amsmath,amscd,amsthm,graphicx}
\usepackage[all]{xy}
\theoremstyle{plain}

\newtheorem{theorem}[subsection]{Theorem}

\newtheorem{corollary}[subsection]{Corollary}
\newtheorem{lemma}[subsection]{Lemma}

\theoremstyle{definition}

\newtheorem{cor}[subsection]{Corollary}

\newtheorem{remark}[subsection]{Remark}

\newcommand{\uu}{\cup}
\newcommand{\ii}{\cap}
\newcommand{\UU}{\bigcup}
\newcommand{\II}{\bigcap}

\newcommand{\sci}{\subset}
\newcommand{\es}{\emptyset}
\newcommand{\set}[1]{\{#1\}}

\newcommand{\ga}{\alpha}
\newcommand{\gb}{\beta}

\newcommand{\gd}{\delta}

\newcommand{\gk}{\kappa}

\newcommand{\gm}{\mu}
\newcommand{\gn}{\nu}
\newcommand{\go}{\omega}

\newcommand{\gq}{\theta}
\newcommand{\gs}{\sigma}

\newcommand{\tit}{\textit}

\newcommand{\C}[1]{\mathcal{#1}}
\newcommand{\D}[1]{\mathbb{#1}}

\newcommand{\te}{\text}

\newcommand{\ol}{\overline}
\newcommand{\ul}{\underline}

\begin{document}

\title{Quantization coefficients in infinite systems}

\author{Eugen Mihailescu and Mrinal Kanti Roychowdhury}


\date{}
\maketitle
\pagestyle{myheadings}\markboth{Eugen Mihailescu and Mrinal Kanti Roychowdhury}{Quantization coefficients in infinite systems}

\begin{abstract}
We investigate quantization coefficients for self-similar probability measures $\mu$ on limit sets, which are generated by systems $\C S$ of infinitely many contractive similarities and by probabilistic vectors. The theory of quantization coefficients for infinite systems has significant differences from the finite case. One of these differences is the lack of finite maximal antichains, and the fact that the set of contraction ratios has zero infimum; another difference resides in the specific geometry of $\C S$ and of its non-compact limit set $J$. We prove that, for each $r\in (0, \infty)$, there exists a unique positive  number $\gk_r $,  so that for any $\kappa < \kappa_r < \kappa'$, the $\gk$-dimensional \textit{lower quantization coefficient} of order $r$ for $\mu$ is positive, and we give estimates for the $\kappa'$-dimensional \textit{upper quantization coefficient} of order $r$ for $\mu$.  In particular, it follows that the quantization dimension of order $r$ of $\mu$ exists, and it is equal to $\kappa_r$. The above results allow then to estimate the asymptotic errors of approximating the measure $\mu$ in the $L_r$-Kantorovich-Wasserstein metric, with discrete measures supported on finitely many points.
\end{abstract}

\section{Introduction and general setting.}

The theory of quantization studies the process of approximating probability measures, which are invariant for certain systems, with discrete probabilities having a finite number of points in their support. Of particular interest are the types of behaviors which may be  encountered in this quantization process for various measures.

Let us consider in general, a probability measure $\mu$ on $\D R^d$, a number $r \in (0, \infty)$ and a natural number $n \in \D N$. Then, the $n$-th \tit{quantization error of order $r$} of $\mu$ is defined by:
\[V_{n, r}(\mu):=\te{inf}\set{\int d(x, \ga)^r d\mu(x): \ga \sci \D R^d, \, \te{card}(\ga) \leq n},\]
where $d(x, \ga)$ denotes the distance from an arbitrary point $x$ to the set $\ga$ with respect to the Euclidean norm on $\D R^d$. If $\int \| x\|^r d\mu(x)<\infty$, then there exists some set $\ga$ for which the infimum is achieved (see \cite{GL1}). A set $\ga$ for which the infimum is achieved is called an \textit{optimal set of $n$-means} or \textit{$n$-optimal set of order $r$}, for the probability $\mu$ and for $0<r<\infty$.

For $s>0$,  the \textbf{$s$-dimensional upper, and lower quantization coefficients of order $r$} for the probability measure  $\mu$, are defined (see \cite{GL1}) respectively as: $$\ol{\mathcal{QC}}_{r, s}(\mu):= \limsup_n n  V_{n,r}(\mu)^{\frac s r}, \ \text{and} \ \ul{\mathcal{QC}}_{r, s}(\mu):= \liminf_n n  V_{n,r}(\mu)^{\frac s r}$$

We will be interested below in quantization coefficients for self-similar probability measures $\mu$ for \textit{infinite} systems of contractive similarities $\C S =(S_1, S_2, \ldots)$ and for infinite probability vectors $p = (p_1, p_2, \ldots)$. In this case, the theory and the techniques of proof from the finite case do not work. In particular, we do not have finite maximal antichains, and also the set of the contraction ratios for the maps $S_i, \ i \ge 1$,  has zero infimum.

\

Recall that in the finite case, a finite self-similar system is determined by a set of contractive similarity mappings on $\D R^d$, namely $\set{S_1, S_2, \cdots, S_N}$  with contraction rates $s_1, s_2, \cdots, s_N$, for $N\geq 2$. By \cite{H} for any probability vector $(p_1, p_2, \cdots, p_N)$ there exists a unique Borel probability measure $\mu$, known as a \textit{self-similar measure}, and a unique nonempty compact fractal subset $J$ of $\D R^d$, which is the support of $\mu$, satisfying the self-similarity conditions:
\begin{equation*} \mu=\sum_{j=1}^N p_j \mu \circ S_j^{-1} \te{ and } J=\UU_{j=1}^N S_j(J)\end{equation*}

The finite iterated system $\set{S_1, S_2, \cdots, S_N}$ satisfies the \textit{open set condition},  if there exists a bounded nonempty open set $U \sci \D R^d$ such that $\UU_{j=1}^N S_j(U) \sci U$ and $S_i(U) \II S_j(U) =\es$ for $1\leq i\neq j\leq N$. The iterated system is said to satisfy the \textit{strong open set condition} if there is an open set $U$ as above,  so that $U\ii J \neq \es$, where $J$ is the limit set of the system (see \cite{H}, etc.)

The \textit{upper, and  lower
quantization dimensions} of order $r$ of $\mu$, are defined  respectively as:
\[\ol D_r(\mu): =\limsup_{n \to \infty} \frac{r\log n}{-\log V_{n, r}(\mu)}; \ \ \ul D_r(\mu): =\liminf_{n \to \infty} \frac{r\log n}{-\log V_{n, r}(\mu)} \]
If $\ol D_r(\mu)$ and $\ul D_r(\mu)$ coincide, we call the common value the \tit{quantization dimension of order $r$} of the probability $\mu$, and is denoted by $D_r(\mu)$. Quantization processes form a rich and far-reaching mathematical concept, with many applications (see for eg. \cite{GG}, \cite{GL1}, \cite{Za}).

Under the open set condition,  Graf and Luschgy (see \cite{GL1, GL2}) showed that the quantization dimension  $D_r(\mu)$ of order $r$ of the probability measure $\mu$ exists, and satisfies the following relation,
$\sum_{j=1}^N(p_j s_j^r)^{{\frac{D_r}{r+D_r}}}=1.$
In fact they proved more, namely that the quantization dimension $D_r(\mu)$ also satisfies the following growth conditions for quantization errors (see \cite{GL3}):
\begin{equation}\label{eq1}  0<\liminf_n n  V_{n,r}(\mu)^{\frac{D_r}{r}} \leq \limsup_n n  V_{n,r}(\mu)^{\frac{D_r}{r}}<\infty\end{equation}
Under the open set condition, Lindsay and
Mauldin (see \cite{LM}) determined the quantization dimension of an $F$-conformal measure $m$ associated
with a conformal iterated function system determined by finitely
many conformal mappings. They established a relationship between the
quantization dimension and the temperature function of the thermodynamic formalism arising in multifractal analysis, and proved that the upper quantization coefficient of $m$ is
finite; however, they left it open whether the lower quantization
coefficient is positive. Using a class of finite maximal antichains Zhu gave an answer in \cite{Z}. Later,
following the same techniques of Lindsay and Mauldin, using the H\"older's inequality, Roychowdhury gave a different proof to show that the lower quantization coefficient for the $F$-conformal measure is positive (see \cite{R1}).

\

In this paper, we are interested in the different case of \textbf{infinite systems} of similarities  $(S_n)_{n\geq 1}$ with similarity ratios
$(s_n)_{n\geq 1}$ respectively, satisfying the strongly separated condition. This setting presents several challenges, different from the finite case. For example in the infinite case, the fractal limit set $J$ of the system is not necessarily compact, by contrast to the finite case. The Hausdorff dimension of the limit set $J$ of an infinite conformal IFS is given in general only as the infimum of the values which make the pressure negative; there may be no zero of that pressure, unlike in the finite case. There are examples of infinite systems where the lower box dimension $\ul{\text{dim}}(J)$ is strictly larger than $HD(J)=h$; and  examples where the Hausdorff measure $H_h(J)$ is zero, while for others $H_h(J)>0$ (see \cite{MaU}). 

Also the boundary at infinity, consisting of accumulation points of sequences of type $(S_i(x_i))_i$ with distinct $i$'s, plays a role in the geometric properties of the respective system, see Mauldin and Urba\'nski \cite{MaU}, \ Mihailescu and Urba\'nski \cite{MU}.
For example, in \cite{MU} it was studied the effect of overlaps and of the boundary at infinity, on the dimensions of limit sets of infinite conformal iterated function systems.

\

For the case of invariant measures for finite or infinite IFS, for which the open set condition is not necessarily satisfied and there may exist overlaps, see Mihailescu and Urba\'nski \cite{MU}, \cite{MU1}, \cite{MU2}.

 Moreover, pertaining to our problem of quantization processes, we do not have finite maximal antichains, and the infimum of the contraction rates is zero, which makes the proofs from the finite case not to work in the infinite situation.

As it turns out, estimating quantization coefficients  in the infinite case is also very different from the finite case. By its intrinsic nature, quantization is a procedure of "fitting" a finite set in the non-compact fractal limit set $J$, in such a way that we obtain as much information as possible about the self-similar measure $\mu$ which is supported on $\ol J$. However, when dealing with an infinite system, usually no finite set $F$ can be placed properly such that every set $S_j(X), j \ge 1$, contains a point from $F$. This makes quantization for infinite systems to be very different than for finite systems.

Let then  $\mu$ be the self-similar probability generated by the system $(S_n)_{n \ge 1}$ and by the probability vector
$(p_n)_{n\geq 1}$ (see for eg. \cite{M}, etc). The measure $\mu$ satisfies the following recursive formula:
\[\mu=\sum_{j=1}^\infty p_j \mu\circ S_j^{-1}\]
The measure $\mu$ is supported on the compact closure $\ol J$ of the associated limit set $J$ (the precise definition will be given below in the General Setting). \

\

We will prove in \textbf{Theorem \ref{theorem}} that, under the strongly separated condition,  for each $r\in (0, \infty)$ there exists a unique $\gk_r\in (0, \infty)$ so that
$\sum_{j=1}^\infty(p_j s_j^r)^{{\frac{\gk_r}{r+\gk_r}}}=1,$ \
and for any $\kappa< \kappa_r$, the $\gk$-dimensional lower quantization coefficient of order $r$ for the self-similar measure $\mu$ satisfies the following \textit{asymptotic condition}:
\[0<\liminf_{n\to\infty} n V_{n,r}(\mu)^{\frac{\gk}{r}}\leq \limsup_{n \to \infty} n V_{n,r}(\mu)^{\frac{\gk}{r}}\]
We also show in \textbf{Theorem \ref{theorem}} that for any $\kappa' > \kappa_r$, the $\kappa'$-dimensional upper quantization coefficient for $\mu$ is finite,
$$ \limsup_{n \to \infty} n V_{n,r}(\mu)^{\kappa'/r} = 0$$
In particular, in \textbf{Corollary \ref{cor-j}} we prove that the \textbf{quantization dimension} of order $r$ of $\mu$ exists and it is equal to $\kappa_r$.
\

In addition, we provide estimates for the upper quantization coefficient $\ol{\mathcal{QC}}_{r, \kappa'}(\mu)$ in the above setting.
As a consequence of the main results, we will prove in Corollary \ref{asy} a result about the asymptotic behavior in $n$, of the approximations in the \textit{$L_r$-Kantorovich-Wasserstein metric} of the self-similar probability  measure $\mu$, by \textbf{discrete} probability measures $Q$ which are supported on $n$ points.

 We also give examples of self-similar measures for infinite systems for which we can obtain estimates on the quantization coefficients.

Some partial attempt to find quantization dimension in infinite IFS was tried in \cite{R2} by Roychowdhury. However, the proof of the main result in that paper is incorrect (see Remark \ref{tare} below); therefore it was not used in the current paper. The ideas and methods in our current paper are completely different.

\

\textbf{General Setting.}

The \textit{$n$-th quantization error} for the probability $\mu$ gives, in essence, the minimal average distance (average with respect to $\mu$), from points in the support of $\mu$ to finite sets of cardinality $n$,  and is defined (see \cite{GL1}) by the formula:
$$ V_{n, r}(\mu):=\te{inf}\set{\int d(x, \ga)^r d\mu(x) : \ga \sci \D R^d, \te{ card}(\ga)\leq n}, $$
and denote  $e_{n, r}(\mu):=V_{n, r}(\mu)^{\frac 1 r}$.
A set $\ga \sci \D R^d$ with card$(\ga) \leq n$ is called an \tit{$n$-optimal set of centers for $\mu$ of order $r$} or \tit{$V_{n, r}(\mu)$-optimal set} whenever we have:
\[V_{n, r}(\mu) =\int d(x, \ga)^r d\mu(x)\]

Let $X$ be a nonempty compact subset of $\D R^d$ with $X=\te{cl(int } X)$. 
Let $(S_j)_{j=1}^\infty$ be an infinite set of contractive similarity mappings on $X$ whose contraction ratios are respectively $(s_j)_{j=1}^\infty$, i.e., $d(S_j(x), S_j(y))=s_j d(x,y)$ for all $x, y \in X$, $0<s_j<1$, $j\geq 1$. We shall assume in the sequel that  $$s:=\sup_{j\ge 1} s_j<1$$

A \textit{word} with $n$ letters in $\D N =\{0, 1, 2, \ldots\}$,  $\go:=\go_1\go_2\cdots \go_n \in \D N^n$, is said to have \textit{length} $n$, for $n\geq 1$. Define also $\D N^{fin}:=\UU_{n\geq 1}\D N^n$ to be the set of finite words with letters in $\D N$, of any length. For $\go=\go_1\go_2\cdots \go_n \in \D N^n$, define:
\[S_\go=S_{\go_1}\circ S_{\go_2}\circ \cdots \circ S_{\go_n} \te{ and } s_\go=s_{\go_1}s_{\go_2}\cdots s_{\go_n}\] The empty word $\es$ is the only word of length $0$ and $S_\es=\te{Id}_X$.
For  $\go \in \D N^{fin}\uu \D N^\infty$ and for a positive integer $n$ smaller than the length of $\go$, we denote by $\go|_n$ the word $\go_1\go_2\cdots \go_n$. Notice  that given $\go \in \D N^\infty$, the compact sets $S_{\go|_n}(X)$, $n\geq 1$, are decreasing and their diameters converge to zero. In fact, we have
\begin{equation} \label{eq113} \te{diam}(S_{\go|_n}(X))=s_{\go_1}s_{\go_2}\cdots s_{\go_n}\te{diam}(X)\leq s^n \te{diam}(X) \end{equation}
Hence for an infinite word $\omega$, the set $\pi(\go):=\II_{n=1}^\infty S_{\go|_n}(X)$
is a singleton, and we can define a map $\pi : \D N^\infty \to X$ which, in view of \eqref{eq113} is continuous. One obtains then the following limit set for the above infinite system of similarities,
\[J:=\pi(\D N^\infty)=\UU_{\go \in \D N^\infty}\II_{n= 1}^\infty S_{\go|_n}(X)\]
This fractal limit set $J$ is \textit{not} necessarily compact in the infinite case, by contrast to the finite case (see \cite{MaU}, \cite{MU}). \
 Let $\gs : \D N^\infty \to \D N^\infty$ be the shift map on $\D N^\infty$, i.e., $\gs(\go)=\go_2\go_3\cdots$ where $\go=\go_1\go_2\cdots$. Note that
$\pi\circ \gs(\go)=S_{\go_1}^{-1}\circ \pi(\go)$, and hence, rewriting $\pi(\go)=S_{\go_1}(\pi(\gs(\go)))$, we see that $J$ satisfies the invariance condition:
\[J=\UU_{i=1}^\infty S_i(J)\]
One says that the above iterated function system satisfies the \tit{open set condition} (OSC) if there exists a bounded nonempty open set $U \sci X$ (in  topology of $X$), so that $S_i(U) \sci U$ for every $i \in \D N$ and $S_i(U) \ii S_j(U)=\es$ for every pair $i, j \in \D N$ with $i\neq j$; and the \tit{strong open set condition} (SOSC) if $U$ can be chosen so that $U \ii J \neq \es$ (see \cite{H}, \cite{M}, etc). Since in our infinite case the limit set $J$ may be non-compact, we will use a stronger condition, namely we say that $\mathcal S$ is \textit{strongly separated} if $\mathcal S$ satisfies the strong open set condition with a bounded open set $U$ and  in addition $d(S_i(U), S_j(U)) >0$ for any $i \ne j$.

\

In the current paper, we assume that the infinite set of similarities is strongly separated.  Notice also that in the infinite systems case,  the open set condition and the strong open set condition are not equivalent, unlike in the finite case (see \cite{SW}).

Let now $(p_1, p_2, \cdots)$ be an infinite probability vector, with $p_j>0$ for all $j\geq 1$. Then there exists a unique Borel probability measure $\mu$ on $\D R^d$ (see \cite{H}, \cite{MaU}, \cite{M}, etc.), such that
\begin{equation*} \label{eq10} \mu=\sum_{j=1}^\infty p_j \mu \circ S_j^{-1}\end{equation*}
This measure $\mu$ is  called the \tit{self-similar measure} induced by the infinite iterated function system of self-similar mappings $(S_j)_{j\geq 1}$ and by the infinite probability vector $(p_1, p_2, \cdots)$, and is obtained as the projection $\pi_*(\nu_{(p_1, p_2, \ldots)})$, where $\nu_{(p_1, p_2, \ldots)}$ is the product measure on $\mathbb{N}^\infty$ induced by $(p_1, p_2, \ldots)$.
One defines the boundary at infinity $\C S(\infty)$ as the set of accumulation points of sequences of type $(S_{i_j}(x_{i_j}))_j$, for distinct integers $i_j$ (see \cite{MaU}).
The self-similar measure $\mu$ is supported in the closure $\ol J$ of the limit set $J$, which is given by $\ol J = J \cup \mathop{\cup}\limits_{\omega \in \D N^{fin}} S_\omega(\C S(\infty))$.\

For the above fixed probability vector $(p_1, p_2, \ldots) $ and contraction vector $(s_j)_{j \ge 1}$, and for arbitrary $q, t \in \mathbb R$, we define the pressure function:
\begin{equation} \label{eq0} P(q, t)=\log \sum_{j=1}^\infty p_j^q s_j^t. \end{equation}

Assume moreover that for every $q \in [0, 1]$, there exists an $u \in \mathbb R$ such that
\begin{equation}\label{5}
0\le P(q, u)<\infty
\end{equation}
In this case, for an arbitrary $q \in \D R$, let $\gq(q)=\inf \set{t \in \D R : \sum_{j=1}^\infty p_j^q s_j^{t} <\infty}$.
Then, for $q\in \D R$ and $t \in (\gq(q), \infty)$, we have $P(q, t)<\infty$. This is similar to the condition of finiteness of entropy in the case of endomorphisms of Lebesgue spaces.

A particular case when the pressure is finite, is when the infinite probability vector $(p_1, p_2, \cdots)$ and the contraction ratios $(s_j)_{j\geq 1}$ satisfy the following \textit{condition}: there exists a constant $a>0$ such that $
\sup_{j}\left|\log p_j- a\log s_j\right|<\infty$.
Then there exists a constant $K\geq 1$ such that for $j\geq 1$,
\begin{equation}\label{eq4001}
 K^{-1} s_j^a \leq p_j\leq K s_j^a
\end{equation}
Condition (\ref{5}) is then satisfied if we have (\ref{eq4001}), since we know that
$s_j^a \le K p_j, \ j \ge 1,$ and since $(p_1, p_2, \ldots)$ is a probability vector, hence for every $q \in [0, 1]$ there exists some $t \in \mathbb R$ such that $\sum_{j=1}^\infty p_j^q s_j^{t} <\infty$. \

The following lemmas are easy to prove.

\begin{lemma} \label{lemma1234} Assuming that condition (\ref{5}) is satisfied above,  it follows that, if $q \in \D R$ is fixed, then the function $t\mapsto P(q, t)$ is strictly decreasing, convex and continuous on $(\gq(q), \infty)$.
\end{lemma}

\begin{lemma}\label{zero}
Assume that condition (\ref{5}) is satisfied. Then for any  $q\in [0, 1]$, there exists a unique $t=\gb(q)\in (\theta(q), \infty)$ such that $P(q, \gb(q))=0$.
\end{lemma}

\begin{proof}
By Lemma~\ref{lemma1234}, for a given $q \in [0, 1]$, the function $P(q, t)$ is strictly decreasing and continuous on $(\gq(q), \infty)$.
Since $0<P(q, u)<\infty$ for some $u\in (\gq(q), \infty)$, in order to conclude the proof it therefore suffices to show that $\lim_{t \to \infty} P(q,t)=-\infty$. For $t > u$,
$$
P(q, t)=\log \sum_{j=1}^\infty p_j^q s_j^t=\log \sum_{j=1}^\infty p_j^q s_j^u s_j^{t-u}
\leq \log \sum_{j=1}^\infty p_j^q s_j^u s^{t-u}=P(q, u) +(t-u) \log s.
$$
Since $s<1$, it follows that $\lim_{t \to \infty} P(q, t) =-\infty$, and thus the lemma is obtained.
\end{proof}

\begin{lemma}
The function $q \mapsto \gb(q)$ given in Lemma \ref{zero}, is strictly decreasing, convex and continuous on $[0, 1]$.
\end{lemma}
\begin{proof}
Let $p=\sup\set{p_1, p_2, \cdots }$. Clearly $p<1$. For any two points $q, q+\gd \in [0, 1]$, where $\gd>0$, we have to show that $\gb(q+\gd) < \gb(q)$. If not let $\gb(q+\gd) \geq \gb(q)$. Then
\begin{align*}
0=P(q+\gd, \gb(q+\gd)) \leq P(q+\gd, \gb(q)) =\log \sum_{j=1}^\infty p_j^{q+\gd} s_j^{\gb(q)} \leq \log \sum_{j=1}^\infty p_j^q p^\gd s_j^{\gb(q)},
\end{align*}
hence $0\leq P(q, \gb(q))+\gd \log p=\gd\log p<0$,
which is a contradiction;  thus $\gb(q+\gd)<\gb(q)$. To show $\gb(q)$ is convex, let $q_1, q_2 \in [0, 1] $ and $a_1, a_2 > 0$ with $a_1+a_2=1$. If $\beta(\cdot)$ is not convex, then there exist $a_1, a_2, q_1, q_2$ such that $\gb(a_1q_1+a_2q_2) > a_1\gb(q_1)+a_2 \gb(q_2)$. Then using H\"older's inequality, we have
\begin{align*}
0&=P(a_1 q_1 +a_2q_2, \gb(a_1q_1+a_2q_2))< P(a_1 q_1 +a_2q_2, a_1\gb(q_1)+a_2 \gb(q_2))\\
&=\log \sum_{j=1}^\infty p_j^{a_1q_1+a_2q_2} s_j^{a_1\gb(q_1)+a_2 \gb(q_2)}\leq \log \Big(\sum_{j=1}^\infty p_j^{q_1}s_j^{\gb(q_1)}\Big)^{a_1} \Big(\sum_{j=1}^\infty p_j^{q_2}s_j^{\gb(q_2)}\Big)^{a_2}\\
& =a_1P(q_1, \gb(q_1)) +a_2P(q_2, \gb(q_2))=0,
\end{align*}
thus contradiction; so $\gb(a_1q_1+a_2q_2) \leq a_1\gb(q_1)+a_2 \gb(q_2)$ i.e., $\gb(q)$ is convex and hence continuous.
\end{proof}

The function $(q, t) \mapsto P(q, t)$ is called the \tit{topological pressure function} corresponding to the given infinite iterated function system. The function $\gb(q)$, sometimes denoted by $T(q)$, is called the \tit{temperature function} (as in \cite{HJKPS}).

\begin{remark}
If $q=0$ then, from \eqref{eq0} we have
$\sum_{j=1}^\infty s_j^{\gb(0)}=1,$ \
i.e., $\gb(0)$ gives the Hausdorff dimension $\te{dim}_\te{H}(J)$ of the infinite self-similar set $J$ (it was shown in \cite{M} that this is the case). Moreover, $P(1, 0)=0$, which gives $\gb(1)=0$.
\end{remark}

\

\section{The quantization coefficients for self-similar measures in the case of infinite systems. }

For arbitrary $r >0$, let us define the auxiliary function $h : (0, 1] \to \D R$ by
$
h(x):=\frac{\gb(x)}{rx}, \ x \in (0, 1]
$, where $\beta(\cdot)$ was defined in Section 1, in terms of the pressure function $P(\cdot)$ of our infinite system.
We know that $\gb(1)=0$ and
  $\gb(0)=\te{dim}_\te{H}(J)$, and so $h(1)=0$ and $\lim_{x\to 0+}
  h(x)=\infty$. Moreover, the function $h$ is continuous and strictly decreasing
   on $(0, 1]$. Hence there exists a unique $q_r \in (0,
  1)$ such that $h(q_r)=1$, i.e.,
$
\gb(q_r)=rq_r,
$  hence $P(q_r, \beta(q_r)) = 0$.
We assume also condition (\ref{5}).
Then, from the above definitions and lemmas it follows that for every $r>0$ there exists a unique number
$\gk_r \in (0, \infty)$, $\kappa_r = \frac{\beta(q_r)}{1-q_r}$, and thus we have the formula
\begin{equation}\label{kr}
P\left(\frac{\gk_r}{r+\gk_r}, \frac{r\gk_r}{r+\gk_r}\right)=0
\end{equation}

\

We now give the main result about quantization coefficients of the self-similar measure $\mu$,  in the infinite system case:

\

\begin{theorem} \label{theorem}
Consider an infinite iterated function system of contractive similarities $\mathcal{S}= (S_1, S_2, \ldots)$ which is strongly separated, and $J$ be its possibly non-compact limit set. Consider the infinite vector $(s_1, s_2, \ldots)$ consisting of the contraction rates of $\mathcal{S}$, and also an infinite probability vector $(p_1, p_2, \cdots)$, such that condition (\ref{5}) above is satisfied. Let us consider $\mu$ to be the self-similar probability measure associated to $\mathcal{S}$ and to $(p_1, p_2, \ldots)$. Denote by $P(q, t)$ the corresponding pressure function, and by $\beta(q)$ the zero of the function $P(q, \cdot)$, and for $r >0$, let $\kappa_r= \frac{\beta(q_r)}{1-q_r}$.

 Then, for any $r
\in (0, \infty)$ and for any $\kappa<  \kappa_r < \kappa'$, the following estimates on the lower/upper quantization coefficients of order $r$ for the self-similar measure $\mu$ (supported on $\ol J$) are true: $$0<\liminf_{n\to\infty} n  V_{n,r}(\mu)^{\gk/r} \ \ \ \text{and} \ \  \ \limsup_{n\to \infty} n  V_{n,r}(\mu)^{\kappa'/r} = 0 $$
\end{theorem}

\begin{proof}

We first want to show that for $\kappa < \kappa_r$, the \textbf{lower quantization coefficient} $\ul{\mathcal{QC}}_{\kappa, r}(\mu)$ is positive, i.e.  that $\liminf\limits_{n\to \infty} n V_{n, r}(\mu)^{\gk/r}>0$, where $\mu$ is the self-similar measure associated to $(S_j)_j $ and to the probabilistic vector $(p_j)_{j \ge 1}$, and where $\gk_r$ is the unique number satisfying the sum condition:
\[\sum_{j=1}^\infty \Big(p_js_j^r\Big)^{\frac{\gk_r}{r+\gk_r}}=1.\]

Let $\tilde \nu$ be the self-similar probability measure corresponding to the infinite system  $(S, \gamma)$ where $S=\set{S_1, S_2, \cdots}$ and $\gamma=(\gamma_1,  \gamma_2, \cdots )$ is the probability vector with $\gamma_j=(p_js_j^r)^{\frac{\gk_r}{r+\gk_r}}, \ j \ge 1$. This measure $\tilde \nu$ can be constructed as the image through the canonical projection $\pi$, of the product measure $\nu_{(\gamma_1, \gamma_2, \ldots)}$ on $\D N^\infty$ associated to the probability vector $(\gamma_1, \gamma_2, \ldots)$; so we have $\tilde \nu = \pi_*(\nu_{(\gamma_1, \gamma_2, \ldots)})$.

Consider  now $U$ to be a bounded open set satisfying the strong separated condition, i.e. $U\II J\neq \es$, and $S_j(U) \sci U$ and $d(S_i(U), S_j(U)) >0$ for $i\neq j$. Then it is easy to show that there exists a finite sequence of integers $\xi$, such that   $J_{\xi}\sci U$, where we denote by $J_\zeta:= S_\zeta(J)$ for arbitrary finite sequence $\zeta$.
Let us take then a finite sequence $\xi$ as above, and define the positive constant $\eta_0:=1-\frac 1 2 \gamma_\xi$. Then, for every nonempty set $V\sci J$ which is open with respect to the induced topology on $J$, it can be proved as in \cite{GL2} that there exists an integer $n\in \D N$ and finite sequences $(\gs^{(k)})_{1\leq k\leq n}$ in $\D N^{fin} \setminus \set{\es}$, such that the sets $J_{\gs^{(1)}}, \cdots, J_{\gs^{(n)}}$ are pairwise disjoint in $V$ and satisfy the following  condition (saying basically that their union has large $\tilde \nu$-measure):
\begin{equation}\label{nuV}
\tilde \nu(V \setminus \UU_{k=1}^n J_{\gs^{(k)}})\leq \eta_0 \cdot \,\tilde \nu(V)
\end{equation}

Moreover, employing the last inequality, one can then show  that
there exists a sequence $(\gs^{(i)})_i$ in $\D N^{fin} \setminus \set{\es}$, such that the associated sets  $J_{\gs^{(i)}}, i \ge 1$ are pairwise disjoint and satisfy:
\begin{equation}\label{infty}
\sum_{i=1}^\infty \tilde\nu(J_{\gs^{(i)}})=1
\end{equation}

We are now ready to prove the lower bound for the quantization coefficients for $\mu$. 
Consider $0<r<\infty$ be fixed and $\gk_r$ as in (\ref{kr}), and let an arbitrary $\kappa< \kappa_r$. Then, we want to show that $\mathop{\liminf}\limits_{n \to \infty} n V_{n,r}(\mu)^{\gk/r} >0$.

By the formula in (\ref{infty}) and from the mutual disjointness of the sets $J_{\gs^{(i)}}, \ i \ge 1$,  we have:
\[1=\sum_{i=1}^\infty \tilde\nu(J_{\gs^{(i)}})=\sum_{i=1}^\infty \Big(p_{\gs^{(i)}}s_{\gs^{(i)}}^r\Big)^{\frac{\gk_r}{r+\gk_r}}\]
However $\frac{\kappa}{r+\kappa} < \frac{\gk_r}{r+\gk_r}$, hence there exists an associated positive integer $m = m(\kappa)$, such that
$$\sum_{i=1}^m \Big(p_{\gs^{(i)}}s_{\gs^{(i)}}^r\Big)^{\frac{\kappa}{r+\kappa}}\geq 1$$\
Now from \cite{GL1} it follows that for every $n\in \D N$, there exists an optimal set $Z_n \sci\D R^d$ with card$(Z_n)\leq n$, such that
\[e_{n, r}^r(\mu) =\int_J d(x, Z_n)^r d\mu(x)\]
Let us define now $\gd_n=\sup_{x \in J} d(x, Z_n)$. Then one has  $\lim_{n\to \infty} \gd_n=0$. But the sets $J_{\gs^{(1)}}, \cdots, J_{\gs^{(m)}}$ are pairwise disjoint and moreover from the strong separated condition $d(S_i(J), S_j(J)) >0$ for any $i \neq j$, therefore we obtain the inequality
\[\gd:=\min\set{d(J_{\gs^{(i)}}, J_{\gs^{(j)}}) : 1\leq i, j\leq m, \, i\neq j}>0\]
Thus, there must exist an integer $n_0 \in \D N$, such that $\gd_n<\frac \delta 2$, for all $n \geq n_0$.
Recalling that $\delta_n=\mathop{\sup}\limits_{x\in J}d(x, Z_n)$, and that $J_{\sigma^{(i)}} \subset J$, we will now look at the subsets of $Z_n$ formed by those points that are closer to $J_{\sigma^{(i)}}$; namely for $n\geq n_0$ and $i\in\set{1, 2, \cdots, m}$, define the set
$Z_{n, i}=\set{a \in Z_n : d(a, J_{\gs^{(i)}}) \leq \gd_n}$. Denote $ k_i(n)=\te{card}(Z_{n, i})$; then
clearly we have $k_i(n) \geq 1$. But, we cannot have a point $x$ in two such sets $Z_{n, i}, Z_{n, j}, i \neq j$ since then it would follow $d(J_{\sigma^{(i)}}, J_{\sigma^{(j)}}) \le 2 \delta_n  < \delta$, thus contradiction with the definition of $\delta$. So the sets $Z_{n, i}$ must be disjoint. Since $Z_{n, i}, i=1, 2, \cdots, m$ are mutually disjoint and contained in $Z_n$ (recall that $Z_n$ has at most $n$ elements), we get that  $\sum_{i=1}^m k_i(n)\leq n$. Hence $k_i(n) \leq n-1, \ i=1, 2, \cdots, m$. As in \cite{GL2} we obtain the inequalities:
\begin{align*} e_{n, r}^r &=\int d(x, Z_n)^r d\mu(x) \geq \sum_{i=1}^m \int_{J_{\gs^{(i)}}} d(x, Z_n)^r d\mu(x)=\sum_{i=1}^m \int_{J_{\gs^{(i)}}} d(x, Z_{n, i})^r d\mu(x)\\
&=\sum_{i=1}^m p_{\gs^{(i)}}s_{\gs^{(i)}}^r \int_{J} d(x, S_{\gs^{(i)}}^{-1}(Z_{n, i}))^r d\mu(x)\geq \sum_{i=1}^m p_{\gs^{(i)}}s_{\gs^{(i)}}^r e_{k_i(n), r}^r
\end{align*}

Define now $\chi=\chi(\kappa)=\min \set{n e_{n, r}^{\gk} : n \leq n_0}$; then $\chi>0$. We show by induction that $\chi \leq n e_{n, r}^{\kappa}$ for $n\ge n_0$. In the induction step, let us assume that $\chi \leq j e_{j, r}^\kappa$ for $j \leq n-1$ and $n-1\geq n_0$. Since $k_i(n) \le n-1$, we can apply the induction step in the last displayed inequality, thus:
$$e_{n, r}^r \geq \sum_{i=1}^m p_{\gs^{(i)}}s_{\gs^{(i)}}^r \chi^{\frac r\kappa}k_i(n)^{-\frac r\kappa}$$
Now, by the generalized H\"older's inequality, we have
$$\sum_{i=1}^m p_{\gs^{(i)}}s_{\gs^{(i)}}^rk_i(n)^{-\frac r \kappa} \geq \Big(\sum_{i=1}^m (p_{\gs^{(i)}}s_{\gs^{(i)}}^r)^{\frac {\kappa}{\kappa+r}}\Big)^{1+\frac r \kappa} \cdot \Big(\sum_{i=1}^m k_i(n)\Big)^{-\frac r\kappa}$$

Recall however that $\sum_{i=1}^m (p_{\gs^{(i)}}s_{\gs^{(i)}}^r)^{\frac {\kappa}{r+\kappa}} \geq 1$ and $\sum_{i=1}^m k_i(n) \leq n$, hence
$e_{n, r}^r \geq \chi^{\frac r\kappa} n^{-\frac r\kappa}$, and $n e_{n, r}^\kappa \geq \chi$. Then by induction, for all $n\ge n_0$, $n e_{n, r}^\kappa \geq \chi>0$. Hence we obtain:
\begin{equation}\label{le}
\liminf_{n\to \infty} n e_{n, r}^{\gk} \geq \chi(\kappa) >0,
\end{equation}
 and therefore  for arbitrary $\kappa< \kappa_r$, the $\kappa$-lower quantization coefficient of order $r$ for $\mu$ is positive.

%
%

\

 Next, we prove  the \textbf{upper bound} of the upper quantization coefficients $\ol{\C{QC}}_{r, \kappa'}(\mu)$ in the infinite self-similar case, for arbitrary  $\kappa'>\kappa_r$. \

Let us first fix an arbitrary number $\kappa> \kappa_r$ and denote by $\eta:= \frac{\kappa}{r + \kappa}$. Then by the definition of $\kappa_r$, we have $\mathop{\sum}\limits_{i \ge 1} (p_i s_i^r)^{\frac{\kappa_r}{r+\kappa_r}} = 1$. So since $\eta > \frac{\kappa_r}{r+ \kappa_r}$, there exists a number $\alpha=\alpha(\eta)$ such that
\begin{equation}\label{al}
\mathop{\sum}\limits_{i \ge 1} (p_i s_i^r)^\eta < \alpha < 1
\end{equation}

Notice now that since $\ol J$ is compact, we can find a finite number of contractive similarities $T_1, \ldots, T_K$ on $X$ such that $S_i(X) \subset T_1(X) \cup \ldots \cup T_K(X), \ i \ge 1$.
Without loss of generality we can assume that all sets $S_j(X)$ are contained in $T_1(X)$ for all $j \ge j_0$, for some large fixed integer $j_0$.
Since $\alpha = \alpha(\eta) < 1$, there exists some integer $N \ge j_0$ such that $$(\mathop{\sum}\limits_{j >N} p_j)^\eta < \frac{1-\alpha}{2}$$ As $\alpha$ depends on $\eta$, also the above integer $N = N(\eta)$ depends on $\eta$. \
Let us define now the finite system of contractive similarities $\tilde S_i, 1 \le i \le N+1$, where $\tilde S_i = S_i, 1 \le i \le N$ and $\tilde S_{N+1} = T_1$, under the above assumption about $T_1$. And define also $\tilde p_i = p_i, 1 \le i \le   N$, $\tilde p_{N+1} = \mathop{\sum}\limits_{i >N} p_i$. We shall denote by $\tilde s_i$ the contraction ratio of $\tilde S_i$, for $ 1 \le i \le N+1$.
Recall that, by our assumption we have $S_i(X) \subset \tilde S_{N+1}(X), \ \forall i >N$.
On the other hand, from the self-similarity condition of the measure $\mu$, we have the decomposition
\begin{equation}\label{dec}
\mu = \mathop{\sum}\limits_{i \ge 1} p_i \mu\circ S_i^{-1} = \mathop{\sum}\limits_{i=1}^N p_i \mu\circ S_i^{-1} + \mathop{\sum}\limits_{j >N} p_j \mu\circ S_j^{-1}
\end{equation}
For $\eta$ and $N$ as above, let us introduce also the following numbers from $(0, 1)$,  $$\gamma_i:= (\tilde p_i \tilde s_i^r)^\eta, \ \ 1 \le i \le N+1$$
Consider now an arbitrary integer $n \ge 2$.
For a finite set $\C F$ of integers, denote by $\C F^*$ the set of all finite sequences of any length, with elements in $\C F$.
For a finite sequence $\omega = (\omega_1, \ldots, \omega_p) \in \{1, \ldots, N+1\}^*, \ p\ge 1$, denote by $\gamma_\omega:= \gamma_{\go_1} \ldots \gamma_{\omega_p}$. Also we denote by $\omega^- = (\omega_1, \ldots, \omega_{p-1})$ to be the truncation of $\omega$ obtained by cutting the last element.

We want now to decompose $\mu$ successively, using (\ref{dec}) up to certain maximal finite sequences $\omega \in \{1, \ldots, N+1\}^*$, until we achieve that all the corresponding  $\gamma_\omega$'s are "almost equal" to $\frac 1n$.
Let us define then the following set of finite sequences determined by $N$ and $n$,
$$F_n:=\{\omega \in \{1, \ldots, N+1\}^*, \ \gamma_\omega \le \frac 1n\cdot \rho(N)^{-1}, \ \gamma_{\omega^-} > \frac 1n\rho(N)^{-1}\},$$where $\rho(N):= \inf\{\gamma_1, \ldots, \gamma_{N+1}\}$.
It follows that if $\omega \in F_n$, then $\gamma_\omega > \frac 1n$.  Also since we assumed that $\tilde p_{N+1}^\eta < \frac{1-\alpha}{2}$ and $\mathop{\sum}\limits_{i=1}^N \gamma_i < \alpha$, and recalling the definition of the $\gamma_i$'s, we obtain
\begin{equation}\label{gamma}
\mathop{\sum}\limits_{i=1}^{N+1} \gamma_i < 1
\end{equation}
Then, recalling that $\gamma_\omega >\frac 1n, \omega \in F_n$, and since we have $1 > \mathop{\sum}\limits_{\omega \in F_n} \gamma_\omega  \ge Card(F_n) \cdot \frac{1}{n}$, we obtain:
\begin{equation}\label{F}
Card(F_n) \le n
\end{equation}
In the identity (\ref{dec}) for $\mu$, we can then continue decomposing successively until reaching the value $\frac 1n$ for $\gamma_\omega$, i.e., we can split $\mu$ according to all finite sequences $\omega \in F_n$. In order to see this, let us deduce from (\ref{dec}) the following decomposition:
\begin{align*}
\mu= &\mathop{\sum}\limits_{i=1}^N p_i\cdot (\mathop{\sum}\limits_{j=1}^N p_j \mu\circ S_j^{-1})\circ S_i^{-1} + \mathop{\sum}\limits_{i=1}^Np_i\cdot (\mathop{\sum}\limits_{j >N} p_j\mu\circ S_j^{-1}) \circ S_i^{-1} + \\
&+\mathop{\sum}\limits_{j >N} p_j\cdot (\mathop{\sum}\limits_{k =1}^N p_k\mu\circ S_k^{-1}) \circ S_j^{-1} + \mathop{\sum}\limits_{j >N}p_j \cdot (\mathop{\sum}\limits_{k >N} p_k \mu\circ S_k^{-1}) \circ S_j^{-1}
\end{align*}
Notice that if a set $B$ has a point in $S_i \tilde S_j(X)$ for some $i, j \in \{1, \ldots, N\}$, then we have $$\int d(x, B)^r d(\mu\circ S_j^{-1} \circ S_i^{-1}) \le s^r_i s^r_j C,$$ for a constant $C >0$. And if $B$ has a point in $S_i S_j(X)$ for some $1\le i \le N$ and $j >N$, then $$\int d(x, B)^r d(\mu\circ S_j^{-1} \circ S_i^{-1}) \le s_i^r \tilde s_{N+1}^r C,$$ since $S_j(X) \subset \tilde S_{N+1}(X)$.
If we take a set $B$ with at least $(N+1)^2$ points such that $B$ has a point in each of the sets $\tilde S_i \tilde S_j, i , j \in \{1, \ldots, N+1\}$, then, since $S_i(X) \subset \tilde S_{N+1}(X), \ i > N$, we obtain the following estimate for the $n$-th quantization error of order $r$ of $\mu$, $$V_{n, r}(\mu) \le C\cdot\big(\mathop{\sum}\limits_{i, j =1}^Np_i p_j s_i^r s_j^r + \mathop{\sum}\limits_{i =1}^N p_i s_i^r (\mathop{\sum}\limits_{j >N} p_j)\tilde s_{N+1}^r + \mathop{\sum}\limits_{j=1}^N p_js_j^r (\mathop{\sum}\limits_{i>N} p_i) \tilde s_{N+1}^r + \mathop{\sum}\limits_{j, k >N} p_jp_k \tilde s_{N+1}^{2r}\big),$$
where $C$ is a positive constant independent of $N$. Similarly we can do this argument for the set $F_n$ instead of $\{1, \ldots, N+1\}$, and we can take a set $B$ of cardinality $n$, which has points in each of the sets $\tilde S_\omega(X)$ for $\omega \in F_n$; this is possible since, as we saw in (\ref{F}), $Card(F_n) \le n$. It follows then similarly as above that
\begin{equation}\label{ue}
V_{n, r}(\mu) \le C\cdot \mathop{\sum}\limits_{\omega \in F_n} \tilde p_\omega \tilde s_\omega^r = C \cdot(\frac 1n)^{\frac{1-\eta}{\eta}} \rho(N)^{\frac{1-\eta}{\eta}}\cdot \mathop{\sum}\limits_{\omega \in F_n} \gamma_\omega \le C \cdot(\frac{\rho(N)}{n})^{\frac{1-\eta}{\eta}}
\end{equation}
Hence recalling that $N$ depends on $\eta$ (hence on $\kappa$), we obtain the following estimate for the $\kappa$-dimensional upper quantization coefficient of order $r$ of $\mu$,
$$\mathop{\limsup}\limits_{n\to \infty} \ n V_{n, r}(\mu)^{\kappa/r} \le C(\kappa)< \infty,$$
where $C(\kappa)$ is a positive constant depending on $\kappa$. In fact if we now take $\kappa'$ arbitrarily larger than $\kappa$ and since $\mathop{\lim}\limits_{n \to \infty}V_{n, r}(\mu)= 0$, we conclude that, for any $\kappa' > \kappa_r$, $$\mathop{\limsup}\limits_{n\to \infty} n V_{n, r}(\mu)^{\kappa'/r} =0$$

\end{proof}

From the above inequalities (\ref{le}) and (\ref{ue}) we obtain also \textit{computable estimates} for the lower and the upper quantization coefficients of order $r$ for the probability measure $\mu$.
We do not know  if the $\kappa_r$-dimensional lower coefficient for $\mu$ of order $r$ is always positive (respectively the $\kappa_r$-upper quantization coefficient of order $r$ being finite) for infinite systems. \
 
In particular, from the estimates above for the $\kappa$-lower/$\kappa'$-upper quantization coefficients of $\mu$ of order $r$, and by taking $\kappa, \kappa' \to \kappa_r$, we obtain that the \textbf{quantization dimension} of order $r$ of $\mu$ exists and is equal to \textbf{$\kappa_r$}. Thus the following result holds:

\

\begin{cor}\label{cor-j}
In the setting of Theorem \ref{theorem}, the quantization dimension $D_r(\mu)$ exists, and $$D_r(\mu) = \kappa_r$$
\end{cor}

\

\begin{remark} \label{tare} \
Notice that, in order  to obtain the proof of Theorem \ref{theorem} and of  Corollary \ref{cor-j}, it is \textbf{not possible} to use finite truncations with $M$ elements $\C S_M$ of the system  and associated self-similar measures $\mu_M$, and then to consider $\log V_{n_k, r}(\mu_M)$ when $n_k\to \infty$, followed by the  use of the estimates for the quantization dimension of $\mu_M$ from the finite case. This problem is due to the fact that the speed of convergence in $n_k$, in the formula for the quantization dimension of $\mu_M$, actually depends on each $M$ (when $M \to \infty$).

$\hfill\square$
\end{remark}


We give now some specific infinite systems, when one can say more about the quantization.

\

\textbf{Examples:}

Consider a sequence of numbers $(s_i)_{i \ge 1}$ in the interval $(0, 1)$, such that $s_i = \gamma^i$, $i \ge 1,$ for some $\gamma \in (0, 1/2)$. Let us also take $p_i = s_i^a = \gamma^{ai}, i \ge 1$ and $p = (p_1, p_2, \ldots)$; in order to make $p$ a probabilistic vector, we will choose $a = \frac{\log 2}{|\log \gamma|}$.

We take then a strongly separated infinite iterated function system $\C S$, formed by the sequence of similarities $\C S = (S_i)_{ i \ge 1}$ of the unit disk $\Delta(0, 1)$ having contraction rates $s_i$ respectively and such that the boundary at infinity $\C S(\infty)$ is equal to the unit circle $S^1$. Consider also the  self-similar probability measure $\mu$, associated to $\C S$ and $p$.  Then, the self-similar measure $\mu$ is supported on the closure $\ol J$, which in this case is given by: $$\ol J = J \cup \mathop{\cup}\limits_{\omega \in \D N^{fin}} S_\omega(\C S(\infty)) = J \cup  \mathop{\cup}\limits_{\omega \in \D N^{fin}}S_\omega(S^1)$$
We notice that in this case $HD(J)  < 1$, but the upper box dimension of $J$  is larger than or equal to 1, since $\ol{\text{dim}}_B(J) \ge \ol{\text{dim}}_B(\C S(\infty)) =1$.
Now, one wants to estimate the quantization coefficients for the measure $\mu$.
According to Theorem \ref{theorem}, the quantization dimension of $\mu$ is equal to $\kappa_r$, where $\kappa_r$ satisfies $$\mathop{\sum}\limits_{i \ge 1} (p_i s_i^r)^{\frac{\kappa_r}{r+\kappa_r}} = 1$$
In our case, the above sum is just the sum for a geometric series, hence we obtain with the above expression for $s_i, p_i$ and the above exponent $a$,
that $$\mathop{\sum}\limits_{i \ge 1} (\gamma^{(a+r)t})^i = 1,$$
where $t =  \frac{\kappa_r}{r+\kappa_r}$. Hence $t = \frac{\log 2}{(a+r)|\log \gamma|} = \frac{\kappa_r}{r+\kappa_r}$.
Therefore, we obtain the quantization dimension $$D_r(\mu) = \kappa_r = \frac{r \log 2}{(a+r) |\log \gamma| - \log 2} = \frac{\log 2}{|\log \gamma|}$$
It is interesting to note that, in this particular case, the quantization dimension $D_r(\mu)$ does not depend on $r$. In general however, if the $p_j$'s are not of the form above, then the quantization dimension $D_r(\mu)$ should depend on $r$. We have also from Theorem \ref{theorem}  that the lower/upper quantization coefficients for $\mu$ satisfy:
    $$ 0 < \mathop{\liminf}\limits_{n\to \infty} n V_{n, r}(\mu)^\frac{\kappa}{r} \ \text{and} \ \mathop{\limsup}\limits_{n\to \infty}  n  V_{n, r}(\mu)^\frac{\kappa'}{r} = 0,  \ \forall \kappa < \log 2/|\log \gamma| < \kappa'$$

\

We notice that this example can be modified so that the images $S_i(\Delta)$ are arranged differently inside $\Delta$, and that the boundary at infinity $\C S(\infty)$ is more complicated, for instance we can imagine an example where it is a countable union of concentric circles $C_n, n \ge 1,$ centered at 0, with radii $c_n$ going to 0.
The corresponding self-similar measure $\mu$ will then be supported on the closure of the limit set $J$, namely on the compact set
$$\ol J = J \cup \mathop{\cup}\limits_{\omega \in \D N^{fin}} S_\omega\big(\mathop{\cup}\limits_n C_n \cup \{0\}\big)$$ Still, if we keep the same contraction rates $s_i$ and the probability vector $p = (p_1, p_2, \ldots)$ as before, then we will obtain the same quantization dimension $\kappa_r$ and quantization coefficients estimates as above.

$\hfill\square$

\

We want now to approximate the self-similar measure $\mu$ with discrete measures of finite support.
Denote by $\C M$ the set of probability measures on the compact set $X \subset \mathbb R^d$. Then,
\begin{align*} d_H (\mu, \gn) :=\sup \Big\{ \Big | \int_X g d\mu- \int_X g d\gn \Big | : \te{Lip } g\leq 1\Big\}, \ (\mu, \gn) \in \C M \times \C M, \end{align*} defines a metric on $\C M$. Then
$(\C M, d_H)$ is a compact metric space (see  \cite{B}). It is known that the $d_H$-topology and the weak topology, coincide on the space of probabilities with compact support (see \cite{Mat}).  In our case all measures are compactly supported.

First, since $X$ is compact we have $\int\|x\|^rd\gm(x)<\infty$, for any probability measure $\mu$ on $X$. For $r \in (0, \infty)$ and for two arbitrary  probabilities $\mu_1, \mu_2$, the \textbf{$L_r$-Kantorovich-Wasserstein  metric}  is defined  by the following formula (see for eg. \cite{GL1}):
\[\rho_r(\mu_1, \mu_2)=\inf_\gn \left(\int \|x-y\|^r d\nu(x, y)\right)^{\frac 1 r},\]
where the infimum is taken over all Borel probabilities $\nu$ on $X \times X$ with fixed \textit{marginal measures} $\mu_1$ and $\mu_2$, such that $(\pi_1)_*(\nu) = \mu_1$ and $(\pi_2)_*(\nu) = \mu_2$ for the canonical projections $\pi_1, \pi_2$ on the first, respectively second coordinates. \

Note that the weak topology, the topology induced by $d_H$, and the topology induced by $L_r$-metric $\rho_r$, \textit{all coincide} on the space $\C M$ (see for example \cite{Ru}).
Let us notice also that, for $r=1$, the $\rho_1$ metric is in fact equal to the $d_H$ metric in the compact case, as  shown by Kantorovich (see \cite{GL1}).

The next Lemma relates the quantization errors for a probability measure $P$, to the $L_r$-Kantorovich-Wasserstein distances between $P$ and discrete measures:

\begin{lemma} (\cite[Lemma~3.4]{GL1})  \label{lemma345}
Let $\C P_n$ denote the set of all discrete probability measures $Q$ on $X$ with $|\te{supp}(Q)|\leq n$. Then for any probability $P$, we have:
$$V_{n, r}(P)=\inf_{Q\in \C P_n} \rho_r^r(P, Q)$$
\end{lemma}

Now by using Lemma \ref{lemma345} and Theorem \ref{theorem}, we obtain the following result about the asymptotic behavior in $n$, of the \textbf{approximations} in $L_r$-metric of $\mu$, with \textbf{discrete} measures supported on $n$ points, when $n$ increases to $\infty$.

\begin{corollary}\label{asy}
In the setting of Theorem \ref{theorem}, let us consider the associated self-similar probability measure $\mu$. Then,  for every $r \in (0, \infty)$, there exists a unique number $\kappa_r \in (0, \infty)$ such that for arbitrary $\kappa, \kappa'$ with  $\kappa<\kappa_r< \kappa'$, the $L_r$-approximations of $\mu$ with discrete measures on $n$ points behave asymptotically as:
$$0<\liminf_{n\to\infty} n^{\frac{1}{\kappa}} \cdot \inf_{Q\in \C P_n} \rho_r(\mu, Q), \ \text{and} \  \limsup_{n \to \infty} n^{\frac{1}{\kappa'}} \cdot \inf_{Q\in \C P_n} \rho_r(\mu, Q) =0$$

\end{corollary}

\


\

\

\  Eugen Mihailescu, \ \

Institute of Mathematics "Simion Stoilow" of the Romanian Academy, Calea Grivitei 21, Sector 1, Bucharest, Romania. \ \ \ \



Email: Eugen.Mihailescu\@@imar.ro \ \ \ \ \
Webpage: www.imar.ro/$\sim$mihailes
\

\

\ Mrinal Kanti Roychowdhury, \ \

Department of Mathematics,
The University of Texas-Pan American,
1201 West University Drive, Edinburg, TX 78539, USA

Email: roychowdhurymk@utpa.edu \ \ \ \ Webpage: http://faculty.utpa.edu/roychowdhurymk

\end{document}